\documentclass[12pt, oneside]{article}    	
\usepackage[a4paper,bindingoffset=0.2in,left=1in,right=1in,top=1in,bottom=1in,footskip=.25in]{geometry}             		
\usepackage{graphicx}	
\usepackage{enumerate}			
\usepackage{enumitem}
\usepackage{amssymb}
\usepackage{amsmath, amsthm, amssymb}
\newtheorem{thm}{Theorem}[section]
\newtheorem{conj}[thm]{Conjecture}

\newtheorem{claim}{Claim}
\newtheorem{lemma}[thm]{Lemma}
\newtheorem{cor}[thm]{Corollary}

\usepackage{tikz}
\newcounter{casenum}
\newenvironment{caseof}{\setcounter{casenum}{1}}{\vskip.5\baselineskip}
\newcommand{\case}[2]{\vskip.5\baselineskip\par\noindent {\bfseries Case \arabic{casenum}:} #1\\#2\addtocounter{casenum}{1}}

\newcommand*{\rom}[1]{\expandafter{\romannumeral #1\relax}}

\usepackage{authblk}
\makeatletter
\def\@cite#1#2{{\normalfont[{\bfseries#1\if@tempswa , #2\fi}]}}
\makeatother

\title{On the number of generalized Sidon sets}

\author{J\'ozsef Balogh\thanks{Department of Mathematical Sciences, University of Illinois at Urbana-Champaign, IL, USA, and Moscow Institute of Physics and Technology, 9 Institutskiy per., Dolgoprodny, Moscow Region, 141701, Russian Federation.} \ \ \ \ \ \ \ Lina Li\thanks{Department of Mathematics, University of Illinois at Urbana-Champaign, Urbana, Illinois 61801, USA. Email: linali2@illinois.edu}}

\begin{document}
\maketitle

\begin{abstract}
A set $A$ of nonnegative integers is called a Sidon set if there is no Sidon 4-tuple, i.e., $(a,b,c,d)$ in $A$ with $a+b=c+d$ and $\{a, b\}\cap \{c, d\}=\emptyset$. Cameron and Erd\H os proposed the problem of determining the number of Sidon sets in $[n]$. Results of Kohayakawa, Lee, R\" odl and Samotij, and Saxton and Thomason has established that the number of Sidon sets is between $2^{(1.16+o(1))\sqrt{n}}$ and $2^{(6.442+o(1))\sqrt{n}}$. An $\alpha$-generalized Sidon set in $[n]$ is a set with at most $\alpha$ Sidon 4-tuples. One way to extend the problem of Cameron and Erd\H os is to estimate the number of $\alpha$-generalized Sidon sets in $[n]$. We show that the number of $(n/\log^4 n)$-generalized Sidon sets in $[n]$ with additional restrictions is $2^{\Theta(\sqrt{n})}$. In particular, the number of $(n/\log^5 n)$-generalized Sidon sets in $[n]$ is $2^{\Theta(\sqrt{n})}$. Our approach is based on some variants of the graph container method.

\end{abstract}

\section{Introduction}\label{intro}
A set $A$ of nonnegative integers is called a \textit{Sidon set} if there is no 4-tuple $(a,b,c,d)$ in $A$ with $a+b=c+d$ and $\{a, b\}\cap \{c, d\}=\emptyset$. Such a tuple $(a, b, c, d)$ is referred to as a \textit{Sidon 4-tuple}. A famous problem raised by Sidon asks the maximum size $\Phi(n)$ of Sidon subsets of $[n]$. Previous studies of Erd\H os and Tur\'an \cite{ET}, Singer \cite{JS}, Erd\H os \cite{PE}, and Chowla \cite{SC}, have showed that $\Phi(n)=(1+o(1))\sqrt{n}$. 
We denote by $\mathcal{Z}_n$ the family of Sidon subsets in $[n]$. 
Cameron and Erd\H os \cite{CE} first proposed the problem of determining $|\mathcal{Z}_n|$.
The extremal result indicates a trivial bound 
\begin{equation}\label{trib}
2^{\Phi(n)}\leq |\mathcal{Z}_n|\leq \sum_{1\leq i\leq\Phi(n)}\binom{n}{i}\leq n^{(1/2+o(1))\sqrt{n}}.
\end{equation}
Cameron and Erd\H os \cite{CE} improved the lower bound by showing $\limsup_n|\mathcal{Z}_n|2^{-\Phi(n)}=\infty$ and asked if the upper bound could also be improved. 
Based on the method introduced by Kleitman and Winston \cite{kleitmanwinston}, Kohayakawa, Lee, R\" odl and Samotij \cite{KLRS} strengthened the upper bound to $2^{c\Phi(n)}$, where $c$ is a constant arbitrarily close to $\log_2(32e)\approx6.442$ for sufficiently large enough $n$.
Using the hypergraph container method~\cite{BMS, ST}, Saxton and Thomason \cite{ST} showed that there are between $2^{(1.16+o(1))\sqrt{n}}$ and $2^{(55+o(1))\sqrt{n}}$ Sidon subsets of $[n]$, which indicates that neither of the bounds in (\ref{trib}) is tight.

We consider counting sets in which a positive upper bound is imposed on the number of Sidon 4-tuples.
An \textit{$\alpha$-generalized Sidon set} in $[n]$ is a set with at most $\alpha$ Sidon 4-tuples. 
One way to extend the Cameron and Erd\H os problem is to estimate the number of $\alpha$-generalized Sidon sets.  Clearly, a trivial lower bound of $2^{\Omega(\sqrt{n})}$ can be given by the number of Sidon sets. 
In this paper, we focus on the case when $\alpha$ is small. In particular, we are interested in determining how large can $\alpha$ be such that the number of $\alpha$-generalized Sidon subsets in $[n]$ is still $2^{\Theta(\sqrt{n})}$. 

For a set $I\subseteq[n]$ and a vertex $v\in [n]$, let $S_I(v)$ be the set of Sidon 4-tuples in $I$ containing $v$ and write $s_I(v)=|S_I(v)|$. Denote by $\mathcal{I}_n(\alpha)$ the family of $\alpha$-generalized Sidon sets $I$ in $[n]$ with $|I|\leq \sqrt{n}/\log n$ or $|I|\geq\sqrt{n}/\sqrt{\log n}$,
and $\mathcal{J}_n(\alpha)$ the family of $\alpha$-generalized Sidon sets $I$ in $[n]$ with $|\{v\in I: s_I(v) \geq \sqrt{n}/\log^4 n\}|\leq\sqrt{n}/\log n$. 
The following are the main results of this paper.
\begin{thm}\label{sidonthm}
Let $\alpha=n/\log^4n$.  For $n$ sufficiently large, we have $|\mathcal{I}_n(\alpha)|\leq 2^{180\sqrt{n}}.$
\end{thm}

\begin{thm}\label{sidonthm2}
Let $\alpha=n/\log^4 n$. For $n$ sufficiently large, we have $|\mathcal{J}_n(\alpha)|\leq2^{180\sqrt{n}}.$
\end{thm}
One can indeed run the same proofs and show that for any given number $c>0$, both theorems hold for $\alpha=cn/\log^4 n$  with the upper bound $2^{C \sqrt{n}}$ for some constant $C$, depending on $c$. 

Theorem~\ref{sidonthm2} immediately implies the following.
\begin{cor}
For $\alpha=O(n/\log^5 n)$, the number of $\alpha$-generalized Sidon sets in $[n]$ is $2^{\Theta(\sqrt{n})}.$
\end{cor}
A simple probabilistic argument can be used to give a lower bound on the number of $\alpha$-generalized Sidon sets in $[n]$: let $m=(\alpha n)^{\frac14}$; a typical $m$-element subset on $[n]$ contains about $\Theta(m^4/n)=\Theta(\alpha)$ Sidon 4-tuples, and there are $2^{\Theta(m\log n)}=2^{\Theta((\alpha n)^{\frac14}\log n)}$ of them. In particular, for $\alpha\gg \sqrt{n}/\log^4 n$, there are $2^{\Theta((\alpha n)^{\frac14}\log n)}\gg 2^{\Theta(\sqrt{n})}$ subsets with $\Theta(\alpha)$ Sidon 4-tuples. Therefore, if the number of $\alpha$-generalized Sidon subsets in $[n]$ has magnitude $2^{\Theta(\sqrt{n})}$, then the order of $\alpha$ cannot be greater than $n/\log^4 n$. We believe that 4 in the exponent is the best possible. 
\begin{conj}\label{conj}
For $\alpha=\Theta(n/\log^4 n)$, the number of $\alpha$-generalized Sidon sets in $[n]$ is $2^{\Theta(\sqrt{n})}.$
\end{conj}

The idea of our proofs is based on the graph container method, in which we assign a \textit{cerfiticate} to each set $I$ in $\mathcal{I}_n(\alpha)$ (or $\mathcal{J}_n(\alpha)$) such that $I$ is contained in a unique `container' determined by its certificate. The certificate should be sufficiently small so that the total number of certificates is properly bounded. Moreover, for each certificate, the number of sets $I$ assigned to it should not be large. Then we can estimate the size of $\mathcal{I}_n(\alpha)$ (or $\mathcal{J}_n(\alpha)$) by counting their certificates. Note that the classical graph container method only applies for the independent sets while we study on the sets with sparse structure. Therefore, we need to make some modifications of the argument. A closely related problem was studied in \cite{BL}, where the authors give an estimate on the number of graphs which contains only few 4-cycles.

Although we did not manage to achieve our goal in this paper, i.e., to prove Conjecture~\ref{conj}, our proof still contains a few new ideas which might be useful to attack some other problems.
The paper is organized as follows. In Section~\ref{superprob}, we present a supersaturation lemma and some probabilistic results to be used in Section~\ref{certifi}. In Section~\ref{certifi}, we introduce our certificate lemmas, Lemmas~\ref{mainlemmasidon} and \ref{mainlemmasidon2}, which are used to prove Theorem~\ref{sidonthm} and \ref{sidonthm2} respectively. The proofs of Theorems \ref{sidonthm} and \ref{sidonthm2} are given in Section~\ref{thmproof}. Finally, we have some concluding remarks in Section~\ref{remark}.
Throughout the paper, we omit all floor and ceiling signs whenever these are not crucial. All logarithms have base 2.

\section{Supersaturation and probabilistic tools}\label{superprob}
\subsection{Supersaturation}
For two sets $A, U\subseteq [n]$,  define a multigraph $H^U(A)$ on vertex set $A$ such that for every $a_1,a_2\in A$ with $a_1<a_2$, the multiplicity of the edge $a_1a_2$ in $H^U(A)$ is the number of ordered pairs $(u_1, u_2)$ in $U$ such that $(a_1, u_1, u_2, a_2)$ is a Sidon 4-tuple. We shall use the following simple supersaturation result.

\begin{lemma}\label{sidonedge}
Let $A, U\subseteq[n]$. If $|A|\cdot|U|\geq6n$, then $e(H^U(A))>\frac{|A|^2|U|^2}{12n}.$
\end{lemma}
\noindent\textit{Proof.}
Let $F$ be a simple bipartite graph defined on the set $A\cup[2n]$ satisfying that for every $a\in A$ and $m\in [2n]$, $a$ is adjacent to $m$ if and only if there is an element $u\in U$ such that $a+u=m$. Clearly, for every vertex $a\in A$, we have $d_F(a)=|U|$. 

Let $\mathcal{P}$ be the set of  paths of length 2 (or 3-paths) in $F$ with endpoints in $A$. Then we have $$|\mathcal{P|}=\sum_{m\in[2n]}\binom{d_F(m)}{2}\geq 2n\binom{\frac{\sum_{m\in[2n]}d_F(m)}{2n}}{2}=2n\binom{\frac{|A|\cdot|U|}{2n}}{2}>\frac{|A|^2|U|^2}{6n}.$$
A path $P=\{xyz\}\in \mathcal{P}$ is called \textit{trivial} if $x+z=y$; otherwise, $P$ is \textit{nontrivial}. Note that $P$ is trivial if and only if both $x$ and $y$ belong to $A\cap U$. Thus, the number of trivial paths in $\mathcal{P}$ is exactly $\binom{|A\cap U|}{2}.$
Let $P'$ be the set of nontrivial paths in $P$. Every 3-path in $P'$ corresponds to an edge in $H^U(A)$ and vice versa. Therefore, we obtain 
\[
e(H^U(A))=|P'|=|P|-\binom{|A\cap U|}{2}>\frac{|A|^2|U|^2}{6n}-\frac{|A|\cdot|U|}{2}\geq\frac{|A|^2|U|^2}{12n},
\]
where the first inequality is given by $|A\cap U|\leq\min\{|A|, |U|\}\leq\sqrt{|A|\cdot|U|}$ and the second inequality follows from the assumption $|A|\cdot|U|\geq 6n$.
\qed
\begin{cor}\label{supersatu}
Let $A\subseteq[n]$ be a set with at most $3n$ Sidon 4-tuples. Then $|A|< \sqrt{6n}.$
\end{cor}
\begin{proof}
Apply Lemma~\ref{sidonedge} with $U=A$. Then we obtain that the number of Sidon 4-tuples in $A$ is more than $|A|^4/12n$. On the other hand, the assumption states that there are at most $3n$ Sidon 4-tuples, which indicates that $|A|^4/12n<3n$, i.e., $|A|<\sqrt{6n}$.
\end{proof}
\begin{lemma}\label{sidonmulti}
Suppose $I, A\subseteq[n]$ and $g\leq n$. For every set $U\subseteq \{v\in I \mid s_I(v)<g\}$ and edge $ab\in H^U(A)$, the multiplicity of $ab$ in $H^U(A)$ is at most $g$.
\end{lemma}
\begin{proof}
Let $m$ be the multiplicity of the edge $ab$ in $H^U(A)$. By the definition of $H^U(A)$, there exist $u_1, u_2, \ldots, u_m$, $v_1, v_2, \ldots, v_m\in U$ such that $a+u_i=v_i+b$, for every $i\in [m]$. Then for every $i\in[m]\setminus\{1\}$, we have $u_i-v_i=b-a=u_1-v_1$, i.e., $u_1+v_i=u_i+v_1$. Since $u_1\in U\subseteq \{v\in I \mid s_I(v)<g\}$, we must have $m-1\leq s_U(u_1)\leq s_I(u_1)<g$, that is, $m\leq g$.
\end{proof}
%

\subsection{Large deviations for sum of partly dependent random variables}
The classical Chernoff bound is a powerful tool, but it only applies to sums of random variables that are independent.
Janson~\cite{J} extended a method of Hoeffding and obtained strong large deviation bounds for sums of dependent random variables with suitable dependency structure.
For a family of random variables $\{Y_{\alpha}\}_{ \alpha\in\mathcal{A}}$, a \textit{dependency graph} is a graph $\Gamma$ with vertex set $\mathcal{A}$ such that if $\mathcal{B}\subset \mathcal{A}$ and $\alpha\in \mathcal{A}$ is not connected by an edge to any vertex in $\mathcal{B}$, then $Y_{\alpha}$ is independent of $\{Y_{\beta}\}_{ \beta\in\mathcal{B}}$.
Let $\Delta(\Gamma)$ denote the maximum degree of $\Gamma$ and let (for convenience) $\Delta_1(\Gamma):=\Delta(\Gamma)+1$.
\begin{thm}[\cite{J}, Corollary 2.2]\label{Jbound}
Suppose that $X$ is a random variable which can be written as a sum
$$X=\sum_{\alpha\in\mathcal{A}}Y_{\alpha},$$
where each $Y_{\alpha}$ is an indicator variable taking the values 0 and 1 only. Let $\Gamma$ be the dependency graph for $\{Y_{\alpha}\}_{ \alpha\in\mathcal{A}}$. Then for $t\geq 0$,
$$\mathbb{P}(X\geq \mathbb{E}[X]+t)\leq \exp\left(-2\frac{t^2}{\Delta_1(\Gamma)|\mathcal{A}|}\right).$$
\end{thm}
\subsection{Some probabilistic results}
For this section, fix $\alpha=\frac{\sqrt{n}}{\log^4 n}$. Let $I$ be an $\alpha$-generalized Sidon set in $[n]$ such that for every $v\in I$, $s_I(v)< \frac{\sqrt{n}}{\log^3 n}$. Define $I_h=\{v\in I: s_I(v)\geq \frac{\sqrt{n}}{\log^4 n}\}$. We further assume that 
\begin{equation}\label{Wassum}
|I|\geq \frac{\sqrt{n}}{\sqrt{\log n}}\quad \text{and} \quad |I_h|> \frac{\sqrt{n}}{\log n}.
\end{equation}

From the Chernoff bound and (\ref{Wassum}), we instantly get the following.
\begin{lemma}\label{Wsize}
Let $W$ be a random subset of $I$ obtained by choosing each $u\in I$ independently with probability $p=\frac{2}{\sqrt{\log n}}$. Then
$\mathbb{P}\left(|W|<\frac{\sqrt{n}}{\log n}\right)=o(1).$
\end{lemma}

For two different numbers $u, v$ and a set $A$, let $S(u, A, v)=\{(u,a,b,v)\mid a, b\in A\text{ and } u+a=b+v \}$ and write $s(u, A, v)=|S(u, A, v)|$. 

 \begin{lemma}\label{Wsmulti}
Let $W$ be a random subset of $I$ obtained by choosing each $u\in I$ independently with probability $p=\frac{2}{\sqrt{\log n}}$. Then almost always $s(u, W, v)\leq 8\frac{\sqrt{n}}{\log^4 n}$, for all $u, v\in I$ simultaneously.
 \end{lemma}
\noindent\textit{Proof.}
It is sufficient to prove the inequality for all $u, v\in I$ with $s(u, I, v)> \frac{8\sqrt{n}}{\log^4 n}$.
For a 4-tuple $r=(u, a, b, v)\in S(u, I, v)$, let $X_r$ be the indicator random variable for the event $r\in S(u, W, v)$. Since $a, b$ are always different, we have $\mathbb{P}(X_r=1)=p^2=\frac4{\log n}$.
Then 
$$\mu_{uv}=\mathbb{E}[s(u, W, v)]=\mathbb{E}\left[\sum_{r\in S(u, I, v)}X_r\right]=p^2s(u, I, v)>\frac{32\sqrt{n}}{\log^5 n}.$$ 
For a given pair of numbers $u, v\in I$, let $\Gamma$ be the dependency graph for $\{X_r: r\in S(u, I, v)\}$. Then we have $\Delta_1(\Gamma)=\Delta(\Gamma)+1\leq 3$.
Using Theorem~\ref{Jbound}, we show that
\begin{equation}\label{chernoffmulti}
\mathbb{P}(s(u, W, v)>2\mu_{uv})<\exp\left(-2\frac{\mu_{uv}^2}{3s(u, I, v)}\right)=\exp\left(-2\frac{\mu_{uv}}{3p^2}\right)<\exp\left(-\frac{16\sqrt{n}}{3\log^4 n}\right).
\end{equation}
On the other hand, by Lemma~\ref{sidonmulti}, we obtain
\begin{equation}\label{meanmulti}
\mu_{uv}=p^2s(u, I, v)\leq p^2\frac{\sqrt{n}}{\log^3 n}=\frac{4\sqrt{n}}{\log^4 n}.
\end{equation}
Combining (\ref{chernoffmulti}) and (\ref{meanmulti}), we obtain
$$\mathbb{P}\left(s(u, W, v)>8\frac{\sqrt{n}}{\log^4 n}\right)<\exp\left(-\frac{16\sqrt{n}}{3\log^4 n}\right).$$
Finally, using the union bound, we have
\[
\pushQED{\qed}
\mathbb{P}\left(\exists u, v\in I \text{ s.t. }s(u, W, v)>8\frac{\sqrt{n}}{\log^4 n}\right)<n^2\exp\left(-\frac{16\sqrt{n}}{3\log^4 n}\right)=o(1).
\qedhere
\popQED
\]

For two sets $B\subseteq A\subseteq [n]$ and a vertex $v\in A$, let $S_{A, B}(v)=\{(a, b, c, d)\in S_A(v) \mid v\in \{a, d\} \text{ and } b, c\in B\}$ and write $s_{A, B}(v)=|S_{A, B}(v)|$. Note that for a Sidon 4-tuple $(a, b, c, d)$, we can switch the $a, b$ and $c, d$ and the resulting tuple is still a Sidon 4-tuple. Therefore, we have $s_{A, A}(v)=\frac12 s_A(v).$
\begin{lemma}\label{Wdegree}
Let $W$ be a random subset of $I$ obtained by choosing each $u\in I$ independently with probability $p=\frac{2}{\sqrt{\log n}}$. Let $S(W)=\{v\in I\mid s_{I, W}(v)>\frac{\sqrt{n}}{\log^4 n}\}$. Then $|S(W)|\leq\frac{16\sqrt{n}}{\log n}$ almost always.
\end{lemma}
\begin{proof}
 Let $R=\bigcup_{v\in I}S_{I, I}(v)$. Then we have 
 \begin{equation}\label{Rsize}
 \frac{n}{\log^4 n}\geq |R|= \frac12\sum_{v\in I}s_{I, I}(v)=\frac14\sum_{v\in I}s_{I}(v)\geq \frac14\sum_{v\in I_h}\frac{\sqrt{n}}{\log^4 n}\geq\frac{n}{4\log^5 n},
 \end{equation}
 where the last inequality holds by (\ref{Wassum}).
Let $R_W=\bigcup_{v\in I}S_{I, W}(v)$. 
 For every $r\in R$, let $X_r$ be the indicator random variable for the event $r\in R_W$. Note that $\mathbb{P}(X_r=1)=p^2$. Then we obtain
$\mathbb{E}[|R_W|]=\mathbb{E}\left[\sum_{r\in R}X_r\right]=p^2|R|.$
Define a simple graph $\Gamma=(R, E)$ such that $$E=\{r_1r_2 \in \binom{R}{2} \mid r_1=(a_1, b_1, c_1, d_1),\ r_2=(a_2, b_2, c_2, d_2)\text{ and }\{b_1, c_1\}\cap\{b_2, c_2\}\neq \emptyset\}.$$
For every $r=(a, b, c, d)\in R$, the number of its neighbors in $\Gamma$ is at most $s_I(b)+s_I(c)< \frac{2\sqrt{n}}{\log^3 n}$, which implies 
\begin{equation}\label{Delta}
\Delta_1(\Gamma)=\Delta(\Gamma)+1\leq \frac{2\sqrt{n}}{\log^3 n}.
\end{equation}
The graph $\Gamma$ can be viewed as the dependency graph of $\{X_r\}_{r\in R}$, since $X_{r_1}, X_{r_2}$ are dependent if and only if $r_1r_2\in E$. 
By Theorem~\ref{Jbound}, we have
\begin{equation*}
\begin{split}
P(R_W\geq 2\mathbb{E}[R_W])&\leq \exp\left(-2\frac{(\mathbb{E}[R_W])^2}{\Delta_1(\Gamma)|R|}\right)=\exp\left(-2\frac{p^4|R|}{\Delta_1(\Gamma)}\right)\\
&\leq \exp\left(-2\frac{ \frac{2^4}{\log^2 n}\cdot\frac{n}{4\log^5 n}}{\frac{2\sqrt{n}}{\log^3 n}}\right)=\exp\left(-\frac{4\sqrt{n}}{\log^4 n}\right),
\end{split}
\end{equation*}
i.e.,
$$|R_W|<2\mathbb{E}[R_W]=2p^2|R|\leq \frac{8n}{\log^5 n}$$
 almost always. Finally, we obtain
$$|S(W)|\leq \frac{2|R_W|}{\frac{\sqrt{n}}{\log^4 n}}\leq \frac{16\sqrt{n}}{\log n}$$
almost always.
\end{proof}

\section{Certificate lemmas}\label{certifi}
In this section, we aim to prove two lemmas which are used to define proper certificates for the desired sets. For the proof of Theorem~\ref{sidonthm}, we introduce Lemma~\ref{mainlemmasidon} as the certificate lemma. A minor modification of its proof gives Lemma~\ref{mainlemmasidon2}, which is used to prove Theorem~\ref{sidonthm2}.
The original proof idea comes from Kleitman and Winston \cite{kleitmanwinston}, who estimated the number of $C_4$-free graphs. Kohayakawa, Lee, R\" odl and Samotij \cite{KLRS} later applied this method to the Sidon problem and gave an upper bound on the number of Sidon sets in $[n]$. 
\begin{lemma}\label{mainlemmasidon}
For a sufficiently large integer $n$, let $\alpha=n/\log^4n$ and $I$ be an $\alpha$-generalized Sidon set in $[n]$ such that for every $v\in I$, $s_I(v)< \sqrt{n}/\log^3 n$. Further assume that the size of $I$ is at least $\sqrt{n}/\sqrt{\log n}$.
Then there exist set sequences $R_0, R_1, \ldots, R_{L}$ and $U_0, U_1, \ldots, U_{L-1}$, where $0\leq L< \log \log n+1$, which determine a unique set sequence $C_0\supset C_1\supset C_2\supset \ldots\supset C_L$. Furthermore, the following are all satisfied:
\begin{enumerate}[label={\upshape(\roman*)}]
 \item $\bigcup_{i=0}^{L} R_i\subseteq I\subseteq C_L\cup \bigcup_{i=0}^{L} R_i$;
  \item $|C_0|\leq n$ and $12\sqrt{n}<|C_i|\leq \frac{6\sqrt{n}\log n}{2^{i-1}}$, for $i=1, 2,\ldots, L-1$;
  \item $R_0\subseteq [n]$ and $|R_0|\leq \frac{16\sqrt{n}}{\log n}$;
  \item $R_i\subseteq C_{i-1}$, $|R_1|\leq\frac{108\sqrt{n}}{\log n}$ and $|R_i|\leq \frac1{2^{2i-4}}\frac{12\sqrt{n}}{\log n}$, for $i=2,\ldots, L$;
 \item $U_0\subseteq [n]$ and $|U_0|=\frac{\sqrt{n}}{\log n}$;
 \item $U_i\subseteq C_i$ and $|U_i|=12\frac{n}{|C_i|}$, for $i=1,\ldots, L-1$;
 \item $L=0$ and $|C_0\cap I|<\frac{\sqrt{n}}{\log n}$ or $|C_L\cap I|<12\frac{n}{|C_L|}$ or $|C_L|\leq 12\sqrt{n}$.
\end{enumerate}
\end{lemma}
We say the set sequences $R_0, R_1, \ldots, R_{L}$ and $U_0, U_1, \ldots, U_{L-1}$ founded in Lemma~\ref{mainlemmasidon} give a \textit{certificate} for $I$. Conditions (\rom{2})--(\rom{5}) guarantee that the number of such certificates is properly bounded. Condition (vii) guarantees that a fixed certificate is associated to small number of sets $I$. This follows from the fact that the most part of $I$ is contained in $C_L$.
\newline

\noindent\textit{Proof of Lemma~\ref{mainlemmasidon}.}
Fix a sufficiently large  integer $n$. Following the ideas of \cite{kleitmanwinston} and \cite{KLRS}, we gave a deterministic algorithm that associates every set $I$ to the desired set sequences. 
\newline 

\noindent\textbf{The core algorithm.} We start with sets $A\subseteq[n]$, $T=\emptyset$ and a function $t(v)=0$, for every $v\in A$. 
Here, one can view $A$ as the set of `available' vertices, $T$ as the set of `selected' vertices, and $t(v)$ as a `state' function which is used to control the process. 
As the algorithm proceeds, we add `selected' vertices from $A$ to $T$ and remove `ineligible' vertices from $A$, whose `state' value exceed some predetermined threshold $t_{\rm{threshold}}$.
More formally, take the auxiliary graph $H$ ($H=H^U(A)$ for some set $U$ and we will discuss the choice of $U$ later) and choose a  vertex $u\in A$ of maximum degree in $H[A]$; we break ties arbitrarily by giving preference to vertices that come early in some arbitrarily predefined ordering. If $u\notin I$, then let $T=T$, $A=A-u$ and $t(v)=t(v)$, for every $v\in A$. Otherwise, let 
  $$t(v)=\left\{\begin{array}{ll} 
    t(v)+d_H(v, u) & \text{for } v\in A,\\
    t(v)  & \text{for } v\notin A,
    \end{array}\right.$$
and define $Q=\{v\in A \mid t(v)>t_{\textrm{threshold}}\}$; let $T=T\cup\{u\}$ and $A=A-u-Q$. We stop the algorithm when $A$ is sufficiently small.
\newline

The goal of the algorithm is to obtain a small representative set $T$ for a given set $I$ such that the choice of $T$ determines a set $A\supseteq I-T$. If $A$ is sufficiently small, then it reduces the number of choices for $I-T$, and hence for $I$.
Note that in each round $T$ increases by at most 1. Therefore, a good algorithm should reduce the size of $A$ rapidly so that we can keep $T$ small in the end.
Recall that in every step, we take a vertex $u$ of maximum degree in the auxiliary graph $H^U(A)$ and add it to $T$ when $u\in I$. After that, we delete `ineligible' vertices, whose `state' exceed the given threshold. The idea behind this is that if the degree of a vertex is larger than the threshold, then it does not belong to $I$, since for every $v\in I$, $s_I(v)$ is bounded.
To speed up the process, we should take a large set $U$ so that we could quickly accumulate the `state' value and produce more `ineligible' vertices in each step. However, the cost of using a larger set $U$ is that the number of choices for $U$ becomes larger and so for the certificates. Therefore, we need to find a balance between the demand for large $U$ and the small number of choices for $U$.
Moreover, ideally if we can find one proper set $U$ through the whole algorithm, then the certificates would be much more concise than in our current lemma. Unfortunately, it turns out that $U$ must vary as the set $A$ shrinks in order to reach the condition of the supersaturation result.

For $i\geq 0$, let $A_i$, $T_i$ and $t_i(v)$ be the state after running the algorithm $i$ rounds. 
In the rest of the proof, we divide the iterations of the core algorithm into several phases and then choose a proper auxiliary set $U$ for each phase. In Phase 1, we execute the algorithm from $A_0=[n]$ to $A_{\ell_1}$, which is the first set $A_i$ of size smaller than $6\sqrt{n}\log n$. For $j\geq 2$, Phase $j$ consists of the executions of the algorithm between $A_{\ell_{j-1}}$, the set produced at the end of Phase $j-1$, and $A_{\ell_j}$, which is the first set $A_i$ of size smaller than $|A_{\ell_{j-1}}|/2$. 
\newline

\noindent\textbf{Set-ups for initial certificate $\mathbf{\{R_0, C_0\}}$.} Let $I_l=\{v\in I: s_I(v)< \frac{\sqrt{n}}{\log^4 n}\}$ and $I_h=\{v\in I: s_I(v)\geq \frac{\sqrt{n}}{\log^4 n}\}$. Based on the size of $I_h$, we have two different set-ups for $R_0$ and $C_0$.
\newline

\noindent\textbf{Case 1.} If $|I_h|\leq \frac{\sqrt{n}}{\log n}$, then we define: 
\begin{equation*}
R_0=I_h, \quad \quad C_0=[n]-R_0.
\end{equation*}

\noindent\textbf{Case 2.} If $|I_h|> \frac{\sqrt{n}}{\log n}$, Lemmas~\ref{Wsize},  \ref{Wsmulti} and \ref{Wdegree} indicate that there exists a set $W\subseteq I$ of size $\frac{\sqrt{n}}{\log n}$ such that 
\begin{equation}\label{Wset}
|S(W)|\leq \frac{16\sqrt{n}}{\log n}
\end{equation} and 
\begin{equation}\label{Wmulti}
s(u, W, v)\leq 8\frac{\sqrt{n}}{\log^4 n},\quad \text{for all }u, v\in I.
\end{equation}
Then we define: $$R_0=S(W), \quad \quad C_0=[n]-R_0.$$
\newline
\noindent\textbf{Phase 1.} If $|C_0\cap I|< \frac{\sqrt{n}}{\log n}$, then we stop the algorithm with $L=0$. Otherwise, take a set $U_0\subseteq [n]$ of size $\frac{\sqrt{n}}{\log n}$: for Case 1, let $U_0$ be an arbitrary subset of $C_0\cap I$ of size $\frac{\sqrt{n}}{\log n}$; for Case 2, let $U_0=W$. Denote $H_0=H^{U_0}(A)$. We now use $H_0$ as an auxiliary graph and run the core algorithm with 
$t_{\textrm{threshold}}=\frac{\sqrt{n}}{\log^4 n}$
 and initial state $$A_0=C_0,\quad T_0=\emptyset \quad \text{and}\quad t_0(v)=0,\;\text{for every}\; v\in A_0,$$ until we obtain the set $A_{\ell_1}$, the first set of size smaller than $6\sqrt{n}{\log n}$.
\newline

Let $K$ be the integer such that $\frac{n}{2^{K}}\leq|A_{\ell_1}|< \frac{n}{2^{K-1}}.$  By the choice of $A_{\ell_1}$, we have $K\leq \frac12 \log n$. 
For every integer $1\leq k\leq K$, let $A^k$ be the first set satisfying 
$\frac{n}{2^{k}}\leq|A^k|< \frac{n}{2^{k-1}}$
 if it exists, $T^k$ be the corresponding $T$-set of $A^k$ and $t^k(v)$ be the corresponding $t$-function. 
Note that $A^k$ may not exist for every $k$. Moreover, $A^{K}$ always exists and it could be $A_{\ell_1}.$
Suppose $$A^{k_1}\supset A^{k_2}\supset\ldots\supset A^{k_{p}},\quad p\leq K\leq \frac12\log n$$ are all the well-defined $A^k$.
From the definition, we obtain that $A^{k_1}=A_0$, $T^{k_1}=T_0=\emptyset$ and $k_p=K$. We additionally define $A^{k_{p+1}}=A_{\ell_1}$ and $T^{k_{p+1}}=T_{\ell_1}$. Then we have $$T_{\ell_1}=T^{k_{p+1}}=\bigcup_{j=2}^{p+1}(T^{k_j}-T^{k_{j-1}}).$$
Now we shall give an estimation on the size of each $T^{k_j}-T^{k_{j-1}}.$

During the process, the algorithm ensures that $t^{k_j}(v)\leq \frac{\sqrt{n}}{\log^4 n}$, for every $v\in A^{k_j}\cup T^{k_j}$. For every $v\in A^{k_{j-1}}-(A^{k_j}\cup T^{k_j})$, suppose $v$ was removed from $A^{k_{j-1}}$ in the $i$-th round and let $u_i$ denote the selected vertex in the round. Then we obtain that 
$$t^{k_j}(v)\leq t_{i-1}(v)+d_{H_0}(v, u_i)\leq \frac{\sqrt{n}}{\log^4 n}+d_{H_0}(v, u_i)\leq\frac{\sqrt{n}}{\log^4 n}+\frac{8\sqrt{n}}{\log^4 n}=\frac{9\sqrt{n}}{\log^4 n},$$ where the last inequality is given by Lemma~\ref{sidonmulti} and (\ref{Wmulti}). Therefore, we have
\begin{equation}\label{ubsidon}
\sum_{v\in A^{k_{j-1}}}t^{k_j}(v)\leq\frac{9\sqrt{n}}{\log^4 n}|A^{k_{j-1}}|<\frac{9\sqrt{n}}{\log^4 n}\frac{n}{2^{k_{j-1}-1}}.
\end{equation}

On the other hand, we can also estimate $\sum_{v\in A^{k_{j-1}}}t^{k_j}(v)$ from the view of `selected' vertices.
Let $2\leq j\leq p$. Take a vertex $u_i\in T^{k_j}-T^{k_{j-1}}$ and suppose that $u_i$ is selected in the $i$-th  round, i.e., from $A_i$. Since $A^{k_j}$ is the first set of size smaller than $\frac{n}{2^{k_{j-1}}}$, we have $|A_i|\geq \frac{n}{2^{k_{j-1}}}$ and then  $|A_i||U_0|\geq 6\sqrt{n}\log n\cdot \frac{\sqrt{n}}{\log n}=6n$. 
By Lemma $\ref{sidonedge}$, we obtain that
$$d_{H_0[A_i]}(u_i)\geq \frac{|A_{i}||U_0|^2}{12n}\geq \frac{n}{12\cdot2^{k_{j-1}}\log^2 n}.$$ 
Since $d_{H_0[A_i]}(u_i)$ does not contribute to $t^{k_j}(v)$ for $v\notin A^{k_{j-1}}$, we have 
 \begin{equation}\label{lbsidon}
 \sum_{v\in A^{k_{j-1}}}t^{k_j}(v)
\geq\sum_{u_i\in T^{k_j}-T^{k_{j-1}}}d_{H_0[A_i]}(u_i)
\geq\left|T^{k_j}-T^{k_{j-1}}\right|\frac{n}{12\cdot2^{k_{j-1}}\log^2 n}.
\end{equation}
Combining (\ref{ubsidon}) and (\ref{lbsidon}), we obtain
$$ \left|T^{k_j}-T^{k_{j-1}}\right|\leq\frac{216\sqrt{n}}{\log^2 n}\quad \text{for }2\leq j\leq p.$$
Let $j=p+1$, since $\frac{n}{2^K}\leq |A^{k_{p+1}}|\leq|A^{k_p}|\leq\frac{n}{2^{K-1}}$, by a similar argument, we obtain that
$$\left|T^{k_{p+1}}-T^{k_{p}}\right|\frac{|A^{k_{p+1}}||U_0|^2}{12n}\leq \sum_{u_i\in T^{k_{p+1}}-T^{k_{p}}}d_{H_0[A_i]}(u_i)\leq \sum_{v\in A^{k_{p}}}t^{k_{p+1}}(v)\leq 9\frac{\sqrt{n}}{\log^4 n}|A^{k_{p}}|,$$
which gives 
$$\left|T^{k_{p+1}}-T^{k_{p}}\right|\leq\frac{216n^{3/2}}{\log^4 n|U_0|^2}=\frac{216\sqrt{n}}{\log^2 n}.$$
We eventually have $$|T_{\ell_1}|=\bigcup_{j=2}^{p+1}|T^{k_j}-T^{k_{j-1}}|\leq \frac12\log n\cdot\frac{216\sqrt{n}}{\log^2 n}=\frac{108\sqrt{n}}{\log n}.$$
\newline
For Phase 1, we define:
$$R_1=T_{\ell_1},\quad \quad C_1=A_{\ell_1}.$$
\newline
\noindent\textbf{Phase 2.} If $|C_{1}\cap I|< 12\frac{n}{|C_1|}$ or $|C_1|\leq12\sqrt{n}$, we stop the algorithm with $L=1$. Otherwise, take an arbitrary set $U_1\subseteq C_{1}\cap I$ of size $12\frac{n}{|C_1|}$ and denote $H_1=H^{U_1}(C_1)$. We will use $H_1$ as an auxiliary graph and run the core algorithm with 
$t_{\textrm{threshold}}=\frac{\sqrt{n}}{\log^3 n}$
 and initial state 
 $$A_0=C_1,\quad T_0=\emptyset\quad \text{and}\quad t_0(v)=0,\;\text{for every}\;v\in A_0,$$
  until we obtain the set $A_{\ell_2}$, the first set of size smaller than $|C_1|/2$.
\newline

We use a similar argument as in Phase 1. For every $v\in A_{\ell_2}\cup T_{\ell_2}$, the algorithm ensures that $t_{\ell_2}(v)\leq \frac{\sqrt{n}}{\log^3 n}$. For every $v\in A_0-(A_{\ell_2}\cup T_{\ell_2})$, suppose $v$ was removed from $A_0$ in the $i$-th round and let $u_i$ denote the selected vertex in the round. Then using Lemma~\ref{sidonmulti}, we obtain that $$t_{\ell_2}(v)\leq t_{i-1}(v)+d_{H_1}(v, u_i)\leq 2\frac{\sqrt{n}}{\log^3 n}.$$ Therefore, we have
\begin{equation}\label{ubsidon2}
\sum_{v\in A_{0}}t_{\ell_2}(v)\leq2\frac{\sqrt{n}}{\log^3 n}|A_{0}|=2\frac{\sqrt{n}}{\log^3 n}|C_1|. 
\end{equation}

On the other hand, take a vertex $u_i\in T_{\ell_2}$ and suppose that $u_i$ is selected in the $i$-th round, i.e., from $A_i$. Since $A_{\ell_2}$ is the first set of size smaller than $|C_1|/2$, we have $|A_i| \geq |C_1|/2$ and then $|A_i||U_1|\geq \frac{|C_1|}{2}\cdot12\frac{n}{|C_1|}=6n$. From Lemma $\ref{sidonedge}$, we obtain that
$$d_{H_1[A_i]}(u_i)\geq \frac{|A_{i}||U_1|^2}{12n}\geq\frac{|C_1||U_1|^2}{24n}.$$
Consequently, we have
 \begin{equation}\label{lbsidon2}
 \sum_{v\in A_0}t_{\ell_2}(v)= \sum_{u_i\in T_{\ell_2}}d_{H_1[A_i]}(u_i)\geq \left|T_{\ell_2}\right|\frac{|C_1||U_1|^2}{24n}.
 \end{equation} 
 Combining (\ref{ubsidon2}) and (\ref{lbsidon2}), we obtain
  $$ \left|T_{\ell_2}\right|\leq \frac{48n^{3/2}}{\log^3 n|U_1|^2}=\frac{48n^{3/2}|C_1|^2}{\log^3 n\cdot12^2n^2}\leq \frac{48n^{3/2}(6\sqrt{n}\log n)^2}{\log^3 n\cdot12^2n^2}=\frac{12\sqrt{n}\log^2 n}{\log^3 n}\leq \frac{12\sqrt{n}}{\log n}.$$
For Phase 2, we define:
$$R_2=T_{\ell_2},\quad \quad C_2=A_{\ell_2}.$$
\newline
\noindent\textbf{Phase $\mathbf{j}$ for $\mathbf{j\geq3}$.} In general, when the algorithm goes to Phase $j$, we first check if $|C_{j-1}\cap I|< 12\frac{n}{|C_{j-1}|}$ or $|C_{j-1}|\leq12\sqrt{n}$. If one of these conditions holds, we stop the algorithm with $L=j-1$. Otherwise, take an arbitrary set $U_{j-1}\subseteq C_{j-1}\cap I$ of size $12\frac{n}{|C_{j-1}|}$ and denote $H_{j-1}=H^{U_{j-1}}(C_{j-1})$. We will use $H_{j-1}$ as an auxiliary graph and run the core algorithm with 
$t_{\textrm{threshold}}=\frac{\sqrt{n}}{\log^3 n}$
 and initial state 
$$A_0=C_{j-1},\quad T_0=\emptyset\quad \text{and}\quad t_0(v)=0,\;\text{for every}\; v\in C_{j-1},$$ until we obtain the set $A_{\ell_j}$, the first set of size smaller than $|C_{j-1}|/2$. Using the exactly same argument as in Phase 2, in the end, we obtain
$$ \left|T_{\ell_{j}}\right|\leq \frac{48n^{3/2}}{\log^3n|U_{j-1}|^2}\leq\frac{48n^{3/2}|C_{j-1}|^2}{\log^3n\cdot12^2n^2}\leq \frac1{2^{2j-4}}\cdot\frac{48n^{3/2}|C_1|^2}{\log^3n\cdot12^2n^2}\leq\frac1{2^{2j-4}}\cdot \frac{12\sqrt{n}}{\log n}.$$
For Phase $j$, we define:
$$R_j=T_{\ell_j}, \quad \quad C_j=A_{\ell_j}.$$
\newline
\indent The algorithm terminates if any of the stopping rules is satisfied. In the process, we obtain set sequences $\{R_0, R_1, R_2, \cdots, R_L\}$, $\{U_0, U_1, U_2, \ldots, U_{L-1}\}$ and $\{C_0, C_1, C_2, \ldots, $ $C_L\}$, which satisfy Conditions (ii)--(vii). From the stopping rules, we know that $12\sqrt{n}< |C_{L-1}|\leq\frac{6\sqrt{n}\log n}{2^{L-2}}$, which implies $L< \log \log n+1.$ 

It remains to check Condition (i). For every $j\geq 0$, if a vertex $v$ was removed in Phase $j$, then there exists $i$ such that $t_i(v) > t_{\textrm{threshold}}$. This implies that there are more than $t_{\textrm{threshold}}$ Sidon 4-tuples containing $v$ in $I$. By the choices of $t_{\textrm{threshold}}$ and $R_0$, we know that $v$ does not belong to $I$, and Condition (i) follows from it.
\qed
\newline

\noindent\textbf{Remark.}  In Case 2, we aim to find a set satisfying inequalities (\ref{Wset}) and (\ref{Wmulti}). For this reason, when we apply the probabilistic method, we need consider the random subset $W\subseteq I$ with the probability $2/\sqrt{\log n}$. On the other hand, the proof requires the size of $W$ to be large enough, i.e., $\sqrt{n}/\log n$. Therefore, it is necessary to assume that $|I|\geq \sqrt{n}/\sqrt{\log n}$.
\newline

Now, let us assume that the set $I$ satisfies $|\{v\in I: s_I(v)\geq \sqrt{n}/\log^4 n\}|\leq \sqrt{n}/\log n$. In regard to this assumption, Case 1 always works for the initial certificate $\{R_0, C_0\}$. This means that when we go through the previous proof under the new assumption, we can actually skip Case 2, where the assumption `$|I|\geq \sqrt{n}/\sqrt{\log n}$' is needed, and let everything else follow in the same way. As a result, we obtain a lemma similar to Lemma~\ref{mainlemmasidon}. (We could get better constants than before, but we do not aim to optimize the constants in this paper.)
\begin{lemma}\label{mainlemmasidon2}
For a sufficiently large integer $n$, let $\alpha=n/\log^4 n$ and $I$ be an $\alpha$-generalized Sidon subset of $[n]$ such that for every $v\in I$, $s_I(v)< \sqrt{n}/\log^3 n$. Further assume that $|\{v\in I: s_I(v)\geq \sqrt{n}/\log^4 n\}|\leq \sqrt{n}/\log n$. Then there exist set sequences $R_0, R_1, \ldots, R_{L}$ and $U_0, U_1, \ldots, U_{L-1}$, where $0\leq L< \log \log n+1$, which determine a unique set sequence $C_0\supset C_1\supset C_2\supset \ldots\supset C_L$. Furthermore, Conditions \emph{(\textit{i})--(\textit{vii})} from Lemma~\ref{mainlemmasidon} are all satisfied.
\end{lemma}

\section{Counting generalized Sidon sets}\label{thmproof}
\noindent\textit{Proof of Theorem~\ref{sidonthm}.} 
Since the number of sets in $[n]$ of size at most $\frac{\sqrt{n}}{\log n}$ is bounded by $2^{\sqrt{n}}$, it is sufficient to count the sets of size at least $\frac{\sqrt{n}}{\sqrt{\log n}}.$
For every $I\in \mathcal{I}_n(\alpha)$, we iteratively remove a number $v$ from $I$, which has $s_I(v)\geq \frac{\sqrt{n}}{\log^3 n}$. Denote by $I'$ the set of remaining numbers. Since $I$ contains at most $\frac{n}{\log^4 n}$ Sidon 4-tuples, the process stops after at most $\frac{\sqrt{n}}{\log n}$ steps, i.e.,
\begin{equation}\label{hset}
|I-I'|\leq \frac{\sqrt{n}}{\log n}.
\end{equation}
This cleaning process ensures that $s_{I'}(v)< \frac{\sqrt{n}}{\log^3 n}$, for every $v\in I'$. By Lemma~\ref{mainlemmasidon}, $I'$ can be associated to a certificate $\{\mathcal{R}, \mathcal{U}\}$, where $\mathcal{R}=\{R_0, R_1, \ldots, R_L\}$ and $\mathcal{U}=\{U_0, U_1, \ldots, U_{L-1}\}$ are two set sequences satisfying Conditions (i)--(vii) in Lemma~\ref{mainlemmasidon}. Thus, each $I\in \mathcal{I}_n(\alpha)$ can be assigned to a certificate $$\mathcal{C}_I=[I-I', L, \mathcal{R}, \mathcal{U}].$$
Note that different sets could have the same certificate. Therefore, to estimate $|\mathcal{I}_n(\alpha)|$, we need to give upper bounds on the number of certificates and on the number of subsets assigned to one certificate.

Let $\mathcal{C}=\{\mathcal{C}_I=[I-I', L, \mathcal{R}, \mathcal{U}] \mid I\in \mathcal{I}_n(\alpha)\}$. For every integer $\ell\geq 0$, denote by $\mathcal{C}_{\ell}$ the set of certificates in $\mathcal{C}$ with $L=\ell$. 
By Lemma~\ref{mainlemmasidon}, we have 
\begin{equation}\label{totalC}
\mathcal{C}=\bigcup_{\ell=0}^{\log \log n+1}\mathcal{C}_{\ell}.
\end{equation}
For $\ell=0$ and a certificate $[I-I',0, \mathcal{R}, \mathcal{U}]\in \mathcal{C}_0$, $\mathcal{U}$ is empty sequence and $\mathcal{R}$ only contains one set, i.e. $\mathcal{R}=\{R_0\}$.  By Lemma~\ref{mainlemmasidon} and (\ref{hset}), $R_0$ and $I-I'$ are subsets of $[n]$ satisfying $|R_0|\leq \frac{16\sqrt{n}}{\log n}$ and $|I-I'|\leq \frac{\sqrt{n}}{\log n}$ respectively. Therefore, the number of certificates in $\mathcal{C}_0$ is 
\begin{equation}\label{c0}
|\mathcal{C}_0|\leq \sum_{i=0}^{\frac{\sqrt{n}}{\log n}}\binom{n}{i}+\sum_{i=0}^{\frac{16\sqrt{n}}{\log n}}\binom{n}{i}\leq2^{16\sqrt{n}+1}.
\end{equation}

For $1\leq\ell\leq\log \log n+1$ and a certificate $[I-I', \ell, \mathcal{R}, \mathcal{U}]\in\mathcal{C}_{\ell},$
$\mathcal{R}, \mathcal{U}$ can be written as $\mathcal{R}=\{R_0, R_1, \ldots, R_{\ell}\}$ and $\mathcal{U}=\{U_0, U_1, \ldots, U_{\ell-1}\}$.
Similarly, since $I-I'\subseteq[n]$ and $|I-I'|\leq \frac{\sqrt{n}}{\log n}$, the number of ways to choose $I-I'$
is at most $$\sum_{i=0}^{\frac{\sqrt{n}}{\log n}} \binom{n}{i}\leq 2\binom{n}{\frac{\sqrt{n}}{\log n}}\leq n^{\frac{\sqrt{n}}{\log n}}= 2^{\sqrt{n}}.$$
%
Now, we discuss the number of choices for sequences $\mathcal{U}=\{U_0, U_1, \ldots, U_{\ell-1}\}$ and $\mathcal{R}=\{R_0, R_1,$ $\ldots, R_{\ell}\}$ iteratively. 
First, by Condition (iii) in Lemma~\ref{mainlemmasidon}, we have $R_0\subseteq [n]$ and $|R_0|\leq \frac{16\sqrt{n}}{\log n}$. Thus, the number of ways to choose $R_0$ is at most 
$$\sum_{i=0}^{\frac{16\sqrt{n}}{\log n}} \binom{n}{i}\leq 2\binom{n}{\frac{16\sqrt{n}}{\log n}}\leq 2^{16\sqrt{n}}.$$
From the proof of Lemma~\ref{mainlemmasidon}, $I-I'$ and $R_0$ determines a unique set $C_0$ of size at most $n$. By Conditions (iv) and (v) in Lemma~\ref{mainlemmasidon}, we obtain that $U_0\subseteq [n]$, $R_1\subseteq C_0$, $|U_0|=\frac{\sqrt{n}}{\log n}$ and $|R_1|\leq \frac{108\sqrt{n}}{\log n}$. Thus, the number of ways to choose $U_0$ and $R_1$ are at most
$$\binom{n}{\frac{\sqrt{n}}{\log n}}\leq 2^{\sqrt{n}}$$ 
and  
$$\sum_{i=0}^{\frac{108\sqrt{n}}{\log n}}\binom{n}{i}\leq 2\binom{n}{\frac{108\sqrt{n}}{\log n}}\leq 2^{108\sqrt{n}}$$ 
respectively.
%
For every $1\leq i\leq \ell-1$, suppose that sets $I-I'$, $R_0, \ldots, R_i$, and $U_0, \ldots, U_{i-1}$ are already fixed.
The proof of Lemma~\ref{mainlemmasidon} shows that there is a unique set $C_i$ such that $U_i, R_{i+1}\subseteq C_i\subseteq[n]$. Moreover, there exists an integer $z_i$ such that 
$$\frac{6\sqrt{n}\log n}{2^{z_i}}<|C_i|\leq \frac{6\sqrt{n}\log n}{2^{z_i-1}},$$
where $1\leq z_1<\ldots<z_{i-1}<z_i<\log \log n$. 
By Conditions (iv) and (vi) in Lemma~\ref{mainlemmasidon}, we obtain that $|U_i|=12\frac{n}{|C_i|}$ and $|R_{i+1}|\leq\frac1{2^{2i-2}}\frac{12\sqrt{n}}{\log n}$. Thus, the number of ways to choose $U_i$ and $R_{i+1}$ are at most 
$$\binom{|C_i|}{12n/|C_i|}\leq \binom{ \frac{6\sqrt{n}\log n}{2^{z_i-1}}}{\frac{12n\cdot2^{z_i}}{6\sqrt{n}\log n}}$$ 
and 
$$\sum_{i=0}^{\frac1{2^{2i-2}}\frac{12\sqrt{n}}{\log n}} \binom{n}{i}\leq 2\binom{n}{\frac1{2^{2i-2}}\frac{12\sqrt{n}}{\log n}}\leq 2^{\frac{12}{2^{2i-2}}\sqrt{n}}$$
respectively.
We summarize the above discussion and obtain that
\begin{equation}\label{cell}
\begin{split}
|\mathcal{C}_l|&
\leq 2^{\sqrt{n}+16\sqrt{n}+\sqrt{n}+108\sqrt{n}}
\cdot\prod_{i=1}^{\infty}2^{\frac{12}{2^{2i-2}}\sqrt{n}}
\sum_{z_1, \ldots, z_{\ell-1}}\binom{ \frac{6\sqrt{n}\log n}{2^{z_1-1}}}{\frac{12n\cdot2^{z_1}}{6\sqrt{n}\log n}}\binom{ \frac{6\sqrt{n}\log n}{2^{z_2-1}}}{\frac{12n\cdot2^{z_2}}{6\sqrt{n}\log n}}\ldots\binom{ \frac{6\sqrt{n}\log n}{2^{z_{\ell-1}-1}}}{\frac{12n\cdot2^{z_{\ell-1}}}{6\sqrt{n}\log n}}\\
&\leq2^{142\sqrt{n}}\cdot\sum_{z_1, \ldots, z_{\ell-1}}\binom{ \frac{6\sqrt{n}\log n}{2^{z_1-1}}}{\frac{12n\cdot2^{z_1}}{6\sqrt{n}\log n}}\binom{ \frac{6\sqrt{n}\log n}{2^{z_2-1}}}{\frac{12n\cdot2^{z_2}}{6\sqrt{n}\log n}}\ldots\binom{ \frac{6\sqrt{n}\log n}{2^{z_{\ell-1}-1}}}{\frac{12n\cdot2^{z_{\ell-1}}}{6\sqrt{n}\log n}},
\end{split}
\end{equation}
where $z_1<z_2<\ldots<z_{\ell-1}$ take over integers in $[1,\ \log\log n).$
To estimate the summation term in inequality (\ref{cell}), we provide the following claim.
\begin{claim}\label{sidonclaim}
For sufficiently large $n$, we have 
$$\binom{ 6\sqrt{n}\log n}{\frac{24n}{6\sqrt{n}\log n}}
\binom{ \frac{6\sqrt{n}\log n}{2}}{\frac{24n\cdot2}{6\sqrt{n}\log n}}
\ldots\binom{ \frac{6\sqrt{n}\log n}{2^{\log\log n-1}}}{\frac{24n\cdot2^{\log\log n-1}}{6\sqrt{n}\log n}}
\leq 2^{25\sqrt{n}}.$$
\end{claim}
\begin{proof}
Let $x=\frac{6\sqrt{n}\log n}{2^{\log\log n-1}}=12\sqrt{n}$. Then the left side is equal to
\begin{equation*}
\begin{split}
\prod_{i=0}^{\log\log n-1} \binom{ 2^ix}{\frac{24n}{2^ix}}&
\leq  \prod_{i=0}^{\log\log n-1}\left(\frac{e\cdot x^22^{2i}}{24n}\right)^{\frac{24n}{2^ix}}
\leq \prod_{i=0}^{\log\log n-1} (6e2^{2i})^{\frac{2}{2^i}\sqrt{n}}\\
 &\leq\prod_{i=0}^{\log\log n-1} 2^{\left[(\log6e+2i)\frac{2}{2^i}\right]\sqrt{n}}
\leq 2^{\left[\sum_{i=0}^{\infty}(\log6e+2i)\frac{2}{2^i}\right]\sqrt{n}}\\
&= 2^{(4\log 6e+8)\sqrt{n}}\leq 2^{25\sqrt{n}},
\end{split}
\end{equation*}
where the first inequality follows from the Stirling's formula.
\end{proof}

Using Claim~\ref{sidonclaim}, we can show that for every $1\leq\ell\leq\log \log n+1$,
\begin{equation}\label{Csub}
\begin{split}
|\mathcal{C}_{\ell}| 
&\leq 2^{142\sqrt{n}}
\cdot\binom{\log\log n}{\ell-1}2^{25\sqrt{n}}
\leq 2^{167\sqrt{n}}\log n.
\end{split}
\end{equation}
Combining (\ref{c0}) and (\ref{Csub}), we obtain
\begin{equation}\label{Cmount}
|\mathcal{C}|=\sum_{\ell=0}^{\log \log n+1}|\mathcal{C}_{\ell}|\leq 2^{16\sqrt{n}+1}+2^{167\sqrt{n}}(\log \log n+1)\log n\leq 2^{168\sqrt{n}}.
\end{equation}

It remains to give an upper bound on the number of subsets assigned to one certificate. For a certificate $C=[I-I', L,  \mathcal{R}, \mathcal{U}]\in\mathcal{C}$, let $\mathcal{I}_{C}=\{I\in \mathcal{I}_n(\alpha) \mid C_I=C\}.$
For every $I\in \mathcal{I}_{C}$, by Lemma~\ref{mainlemmasidon}, we have 
$$I\subseteq (I-I')\cup \bigcup_{i=0}^{L} R_i\cup C_L,$$ 
where $C_L$ is uniquely determined. Note that the set $(I-I')\cup \bigcup_{i=0}^{L} R_i$ is given by the certificate $C$.  Therefore, $\mathcal{I}_{C}$ is decided by the ways to choose $C_{L}\cap I=C_{L}\cap I'$. There are three cases:
\begin{caseof}
  \case{$|C_L|\leq 12\sqrt{n}$.}{In the case, we have $|\mathcal{I}_{C}|\leq 2^{|C_L|}\leq 2^{12\sqrt{n}}.$}
  \case{$L=0$ and $|C_L|> 12\sqrt{n}$.}{By Condition (vii) in Lemma~\ref{mainlemmasidon}, for every $I\in \mathcal{I}_{C}$, $I$ satisfies $|C_0\cap I'|< \frac{\sqrt{n}}{\log n}$. In this case, we have
$$|\mathcal{I}_C|\leq\sum_{i=0}^{\frac{\sqrt{n}}{\log n}}\binom{|C_0|}{i}\leq \sum_{i=0}^{\frac{\sqrt{n}}{\log n}}\binom{n}{i}\leq 2\binom{n}{\frac{\sqrt{n}}{\log n}}\leq 2^{\sqrt{n}}.$$}
  \case{$L\geq 1$ and $|C_L|> 12\sqrt{n}$.}{By Condition (vii) in Lemma~\ref{mainlemmasidon}, for every $I\in \mathcal{I}_{C}$, $I$ satisfies $|C_L\cap I'|<12\frac{n}{|C_L|}$. In this case, we have 
\begin{equation*}
\begin{aligned}
  |\mathcal{I}_C|&\leq\sum_{i=0}^{12\frac{n}{|C_L|}}\binom{|C_L|}{i}\leq 2\binom{|C_L|}{12\frac{n}{|C_L|}}.
\end{aligned}
\end{equation*}
Let $x=12\frac{n}{|C_L|}$. By convexity, we obtain that
$$|\mathcal{I}_C|\leq 2\binom{12n/x}{x}
\leq 2\left(\frac{12en}{x^2}\right)^{x}
\leq 2^{[\log(12en)-2\log x]x+1}
\leq 2^{\sqrt{12en}+1}
\leq 2^{6\sqrt{n}}.$$}
\end{caseof}

From the above discussion, for every $C\in \mathcal{C}$, we have
\begin{equation}\label{Cone}
|\mathcal{I}_C|\leq 2^{12\sqrt{n}}.
\end{equation}
Eventually, combining (\ref{Cmount}) and (\ref{Cone}), we obtain that
\[
\pushQED{\qed}
|\mathcal{I}_n(\alpha)|\leq |\mathcal{C}|\cdot 2^{12\sqrt{n}}\leq 2^{180\sqrt{n}}.
\qedhere
\popQED
\]
\newline

\noindent\textit{Proof of Theorem~\ref{sidonthm2}: }  
For every set $J\in \mathcal{J}_n(\alpha)$, we apply the same cleaning process as in the proof of Theorem~\ref{sidonthm} and obtain a set $J'$ satisfying $|J-J'|\leq \frac{\sqrt{n}}{\log n}$ and $s_{J'}(v)<\frac{\sqrt{n}}{\log^3 n}$ for every $v\in J'$. Due to the definition of $\mathcal{J}_n(\alpha)$ and $J'\subseteq J$, we also have $|\{v\in J': s_{J'}(v) \geq \frac{\sqrt{n}}{\log^4 n}\}|\leq \frac{\sqrt{n}}{\log n}$.
By Lemma~\ref{mainlemmasidon2}, $J'$ can be associated to a certificate $\{\mathcal{R}, \mathcal{U}\}$, where $\mathcal{R}=\{R_0, R_1, \ldots, R_L\}$ and $\mathcal{U}=\{U_0, U_1, \ldots, U_{L-1}\}$ are two set sequences satisfying Conditions (i)--(vii) in Lemma~\ref{mainlemmasidon}. The rest of the proof is the same as that of Theorem~\ref{sidonthm}.
\qed

\section{Concluding remarks}\label{remark}
\noindent\textbf{Remark 1.}
In \cite{ST}, Saxton and Thomason established the hypergraph container theorem not only covering independent sets but also for sufficiently sparse structures. One can use their result to estimate the number of $\alpha$-generalized sets for some functions $\alpha$; however, the estimates obtained from it are weaker than the ones from the graph container method. 
To be more specific, using the hypergraph container method, we would consider the 4-uniform hypergraph whose vertex set is $[n]$ and whose edges are all the Sidon 4-tuples; to generate small containers, we need to iterate Theorem 6.2 (\cite{ST}) repeatedly  $\Theta(\log n)$ times. This produces $2^{O(n\tau\log(1/\tau)\log n)}$ containers of size at most $O(n\tau)$, for the sets with at most $O(\tau^4n^3)$ Sidon 4-tuples.
Since we are interested in obtaining a family of containers with $2^{\Theta(\sqrt{n})}$ elements, the order of $\tau$ should not be higher than  $1/(\sqrt{n}\log^2 n)$. (One can easily check that $\tau=\Theta(1/\sqrt{n}\log^2 n)$ satisfies the conditions of Theorem 6.2.) Therefore, the hypergraph container theorem in \cite{ST} provides that the number of $\alpha$-generalized Sidon is $2^{O(\sqrt{n})}$ for $\alpha=O(\tau^4n^3)=O(n/\log^8 n)$, while the best result we have is for $\alpha=O(n/\log^5 n)$.
\newline

\noindent\textbf{Remark 2.}
We also studied the family of $\alpha$-generalized Sidon sets for some other functions $\alpha$. Denote by $\mathcal{G}_n(\alpha)$ the family of $\alpha$-generalized Sidon sets in $[n]$. The results we have is summarized in the following table.

\begin{table}[h!]
\centering
\begin{tabular}{ p{2cm}||p{5cm}|p{5cm}}
 \hline
 $\alpha$   & Upper bound for $|\mathcal{G}_n(\alpha)|$ & Lower bound for $|\mathcal{G}_n(\alpha)|$\\
 \hline
$n/\log^5 n$   & $2^{O(\sqrt{n})}$    & $ 2^{\Omega(\sqrt{n})}$\\
$n/\log^4 n$   & $2^{O(\sqrt{n}\log^{1/4} n))}$  & $2^{\Omega(\sqrt{n})}$\\
$n/\log^3 n$   & $2^{O(\sqrt{n}\sqrt{\log n})}$ & $2^{\Omega(\sqrt{n}\log^{1/4} n)}$\\
$n/\log^2 n$   & $2^{O(\sqrt{n}\log^{3/4} n)}$ & $2^{\Omega(\sqrt{n}\sqrt{\log n})}$\\
$n/\log n$   &  $2^{O(\sqrt{n}\log n)}$ & $2^{\Omega(\sqrt{n}\log^{3/4} n)}$\\
 $n$ & $ 2^{O(\sqrt{n}\log n)}$  & $2^{\Omega(\sqrt{n}\log n)}$\\
 \hline
\end{tabular}
\caption{The number of $\alpha$-generalized Sidon sets.}
\label{table}
\end{table}

In Table~\ref{table}, all the lower bounds come from the probabilistic argument discussed in Section~\ref{intro}, except for the case $\alpha\leq n/\log^4 n$, where we use the number of Sidon sets as the lower bound; all the upper bounds follow from our graph container method, except for the case $\alpha=n$, where we use Corollary~\ref{supersatu}. For $\alpha\in \{ n/\log^5 n, n\}$, the current bounds are tight. For other $\alpha$, the distance between the lower bound and the upper bound is a $\log^{1/4} n$ factor on the exponent. We believe that the lower bounds are the truth.
\newline

\noindent\textbf{Acknowledgment.} We thank for Wojciech Samotij for some useful comments.

\end{document}